\documentclass{amsart}

\usepackage{amsfonts, amsmath, amsthm, amssymb}
\usepackage[all]{xy}

\newtheorem{theorem}{Theorem}[section]
\newtheorem{lemma}[theorem]{Lemma}
\newtheorem{prop}[theorem]{Proposition}

\newtheorem{definitiontemp}[theorem]{Definition}
\newenvironment{definition}{\begin{definitiontemp}
\normalfont}{\end{definitiontemp}}

\theoremstyle{remark}
\newtheorem*{remark}{Remark}

\newcommand{\cp}{\mathbb{C}_p}
\newcommand{\qp}{\mathbb{Q}_p}
\newcommand{\A}{\mathbb{A}}
\newcommand{\dwcoh}{H^\bullet_{DW}}
\newcommand{\R}{\mathbb{R}}
\newcommand{\C}{\mathbb{C}}
\newcommand{\proj}{\mathbb{P}}

\title{Frobenius map on local Calabi-Yau manifolds}
\author{I. Shapiro}

\begin{document}

\begin{abstract}
We prove results that, for a certain class of non-compact Calabi-Yau threefolds,  relate the Frobenius action on their $p$-adic cohomology to the Frobenius action  on the $p$-adic cohomology of the corresponding curves.  In the appendix, we describe our interpretation of the Griffiths-Dwork method.
\end{abstract}

\maketitle

\section{Introduction}
In the present paper we will consider (families of) non-compact manifolds $\widetilde{X}$ specified by the equation \begin{equation}\label{localeq}xy+P(u,v)=0\end{equation} where $x, y, u, v$ are in $\C$ or $x,y\in\C$, and $u,v\in \C^\times$, and $P$ may depend on a parameter $\lambda$.  (In the first case $P(u,v)$ is a polynomial and in the second case it is a polynomial in $u^{\pm 1}, v^{\pm 1}$.)  In fact we consider more general situations, but it is the Calabi-Yau manifolds of this kind that appear in many physically interesting situations.  Namely, the first case is realized in matrix models and the manifolds of the second kind appear as mirror partners of non-compact toric Calabi-Yau threefolds (local threefolds).

It is well known that many questions related to these manifolds can be reduced to the study of the Riemann surface $P(u,v)=0$ that we will denote by $X$.  It was shown in \cite{ksv, sv2} that the analysis of instantons, mirror map, and the number of holomorphic disks can be reduced to the study of the Frobenius map on the cohomology of the Calabi-Yau threefold considered as a variety over the $p$-adics; more precisely, we change our coefficients from complex numbers $\C$ to the $p$-adic complex numbers $\cp$.

The main goal of the present paper is to relate the Frobenius map on the $p$-adic cohomology of the Calabi-Yau threefold specified by the equation (\ref{localeq}) to the Frobenius map on the $p$-adic cohomology of the Riemann surface $P=0$.  To begin, one must relate the cohomology groups themselves. The existence of this relationship is known and widely used in the physics literature (though not in the $p$-adic setting), however we do not know of a reference containing the relevant proofs.  Thus we begin by proving some results that clarify the precise relation in the cases that we consider.  Roughly speaking,  the cohomology of $X$ is closely related to the shifted (by two) cohomology of $\tilde{X}$, in particular, the simplest case that we examine is $x,y,u,v\in\cp$ and  in this instance, for $i\geq 1$, we have that $H^i(X)\cong
H^{i+2}(\tilde{X})$. The main purpose of the paper is to establish the following:  \begin{equation}\label{main}p\cdot Fr_{X}=Fr_{\tilde{X}}.\end{equation}

We will demonstrate this relation  using the Dwork's definition of the Frobenius map (on Dwork cohomology) or more precisely a version of this definition that we outline below (see \cite{homoint} for more details).  Finally, in the appendix we provide a brief discussion of the modifications of this method necessary to handle the projective case that arises even in the present, seemingly affine, situation.  The discussion is our interpretation of the Griffiths-Dwork method (\cite{coxkatz} and references therein) that we use in \cite{shapiro}.

\section{Preliminaries}
Let $\cp$ denote the completion of the algebraic closure of the
$p$-adic numbers $\qp$.  We assume that $p$ is an odd prime and
$\pi\in\cp$ is such that $\pi^{p-1}=-p.$

Denote by $\cp^\dagger\langle x_i\rangle$ the subring of the formal
power series ring $\cp[[x_i]]$ consisting of the overconvergent series.
More precisely, $\cp^\dagger\langle x_i\rangle$ consists of elements
$\sum a_I x^I$ with $\text{ord}_p a_I\geq c|I|+d$ and $c>0$, i.e., those
power series that converge on a neighborhood of the closed polydisk
of radius $1$ around $0\in\cp^n$.

Note that one can take quotients of $\cp^\dagger\langle x_i\rangle$.
Thus the expression $\cp^\dagger\langle x,x^{-1}\rangle$ is to be
understood as $\cp^\dagger\langle x,y\rangle /\{xy=1\}$ and this
ring consists of elements of the form $\sum_{-\infty}^\infty a_i x^i$
with $\text{ord}_p a_i\geq c|i|+d$ and $c>0$.

We use the language of $D$-modules, however in our case it is sufficient to restrict the attention to $D$-modules on affine spaces.  More precisely, for us a $D$-module $M$ is a module over the Weyl algebra $\mathcal{W}_{z_1,..,z_n}$, i.e., an algebra generated by $z_i$ and $\partial_{z_i}$ subject to the relations $\partial_{z_i} z_j- z_j\partial_{z_i}=\delta_{ij}$.  A pushforward of $M$ onto the first $s$ coordinates $z_1,..,z_s$ is the complex of $\mathcal{W}_{z_1,..,z_s}$-modules given by $M[dz_{s+1},..,dz_n]$ (where $dz_i$ are odd variables of degree $1$ and elements of $M$ are assigned degree $0$) with the differential $\sum_{i=s+1}^n \partial_{z_i}\cdot dz_i$.  This is also known as the relative de Rham complex of the $D$-module $M$.  If $s=0$ then the resulting complex of $D$-modules over a point, i.e., vector spaces, is the de Rham complex of $M$, whose cohomology is the de Rham cohomology of the $D$-module $M$.

The most basic $D$-module that we consider is $\cp^\dagger\langle z_i\rangle$ with the obvious $\mathcal{W}$ action.  A more interesting and key example for us is $\cp^\dagger\langle z_i\rangle e^f$ which is the same as $\cp^\dagger\langle z_i\rangle$ as a vector space, but the action of $\partial_{z_i}$ is modified to act by $\partial_{z_i}+(\partial_{z_i}f)$. Note that one also has a $D$-module $\cp[z_i]e^f$ defined similarly.  When there is no chance of confusion we simply write $e^f$ to denote a $D$-module of this type.  Conversely, we sometimes write $g(z_i)$ to denote the element $g(z_i)e^f$.

A useful formula (see \cite{katz} for example) we will need later is
\begin{equation}\label{theformula}
\sum A_n
\dfrac{(-1)^n}{\pi^n}(pb+n-1)!=\dfrac{\pi^{(p-1)b}(b-1)!}{p}=(-1)^b
p^{b-1}(b-1)!\end{equation} where $\sum A_n \theta^n:=\text{exp}(\pi(\theta^p-\theta))$.
The Dwork's construction of the Frobenius action is based on the remark that $e^{\pi(\theta^p-\theta)}$ is in
$\cp^\dagger\langle \theta\rangle$ (see \cite{katz}), but $e^\theta$
itself is not.

If $C^\bullet$ is a cohomological complex, we denote by $C^\bullet[n]$ the same complex with the grading shifted by $n$.  More precisely, if $D^\bullet=C^\bullet[n]$, then $D^i=C^{i+n}$.

For a collection of operators $D_i:A_i\rightarrow B$, we write $B/D_i$ to denote the quotient of  $B$ by span of the images of
the operators $D_i$, i.e., $B/D_i=\frac{B}{\sum_i D_i(A_i)}$.

\section{De Rham cohomology and the Frobenius map}
In this section we
explain how to define the Dwork cohomology of a hypersurface in an \emph{affine space}\footnote{We deal with the \emph{projective space} in the appendix.} and the Frobenius action on it.  This gives a construction of the Frobenius map on the cohomology of this hypersurface.  The comparison between the approach used here and the one that is now considered usual is carried out in \cite{katz}.  If one neglects the (admittedly important) technicality of overconvergence, then it is not hard to see the relationship between the two (see \cite{homoint} for a direct proof or the appendix of this paper for a more conceptual sketch based on \cite{ddr}).   We remark that the cohomological degree in Dwork cohomology should be shifted down by
two for comparison with the usual (for example de Rham) cohomology.  Also the Frobenius action
defined here \emph{acquires an extra factor of $p$} as compared with
the usual Frobenius map.  To avoid confusion, we state all of our results in the standard convention.

Let $P$ be an equation defining a smooth
hypersurface\footnote{There exist versions of the following for higher
codimension, however we will not discuss them.} $X\subset\A^n$ (loosely speaking $\A^n=\A^n_{\cp}=\cp^n$).  Recall that the $D$-module $\cp^\dagger\langle
z_i\rangle e^{f}$ is defined by letting $\partial_{z_i}$ act via
$\partial_{z_i}+(\partial_{z_i}f)$.
\begin{definition}
The Dwork cohomology of
$X$, $H^\bullet_{DW}(X)$ is defined as the de Rham
cohomology of the
$D$-module $\cp^\dagger\langle z_i,t\rangle e^{\pi Pt}$ on
$\A^n\times\A^1$.
\end{definition}

\begin{remark}Recall that the de Rham cohomology, being the pushforward of
the $D$-module along the projection to a point, can be computed in
stages, via successive pushforwards; this is a direct analogue of
the Fubini theorem.  We take advantage of this fact repeatedly in the rest of the paper.
\end{remark}

The Frobenius action is defined on the $D$-module $\cp^\dagger\langle z_i,t\rangle
e^{\pi Pt}$ as follows:
$$Fr(f(z_i,t)e^{\pi P(z_i)t})=(e^{\pi(P(z_i^p)t^p-P(z_i)t)}f(z_i^p,t^p))e^{\pi
P(z_i)t}$$ where  $e^{\pi(P(z_i^p)t^p-P(z_i)t)}\in
\cp^\dagger\langle z_i,t\rangle$ since
$e^{\pi(\theta^p-\theta)}\in\cp^\dagger\langle \theta\rangle$.  This
is where one needs overconvergent series, they allow enough freedom
to define the natural action of Frobenius as above, yet enlarging
the polynomials by them does not change the cohomology.\footnote{The
issue is that while the de Rham cohomology of $\C[x]$ and $\C[[x]]$
is $\C$, this is no longer true for $\{\sum a_n x^n|a_n\rightarrow
0\}$ because $d(\sum x^{p^n})=\sum p^n x^{p^n-1}$.  This does not
happen if we allow only overconvergent series.  We can not use $\C[x]$ because it does not contain $e^{\pi(x^p-x)}$ and we can not use $\C[[x]]$ because it does contain $e^{\pi x}$.} In this way
$\dwcoh(X)$ acquires an action of $Fr_X$ by naturality.  More precisely, if $\omega$ is an overconvergent differential form, i.e., $\omega(z_i,t)=g(z_i,t)dz_{i_1}..dz_{i_s}dt$ with $g(z_i,t)\in\cp^\dagger\langle z_i,t\rangle$, then $$Fr(\omega)=e^{\pi(P(z_i^p)t^p-P(z_i)t)}g(z^p_i,t^p)dz^p_{i_1}..dz^p_{i_s}dt^p.$$

Note that everything that was discussed in this section still holds if we replace $\A^n$ with $Y=\A^n\times(\A^\times)^m$.  So that the $p$-adic cohomology of a hypersurface in $Y$, specified by an equation $P=0$ with $P$ a polynomial in variables $z_1,..,z_n$ and $w_1^{\pm 1},..,w_m^{\pm 1}$, also acquires a Frobenius action.

\section{From $xy+P=0$ to $P=0$}
This section addresses, in a purely algebraic manner, the comparison between the cohomologies
of the smooth affine hypersurfaces $xy+P=0$ and  $P=0$.  For us $P$ will always be an algebraic function on $Y=\A^n\times(\A^\times)^m$, i.e., $P\in\cp[z_i,w_j^{\pm 1}]$.  One can allow for more general $Y$ as far as the cohomology comparison is concerned, but it is not clear how to define the Frobenius map on them.  Furthermore, we compare the Frobenius actions  which turn out to be compatible as well.

\subsection{Key Lemma}
The Lemma below serves as the backbone of the paper.

\begin{lemma}\label{exptxy}
As a complex of $D$-modules, the pushforward of $\cp^\dagger\langle
t,x,y\rangle e^{\pi txy}$ along the projection onto the first
coordinate is quasi-isomorphic to $$\frac{\cp^\dagger\langle z^{\pm
1}\rangle}{\cp^\dagger\langle z\rangle}[-1]\oplus\cp^\dagger\langle
z^{\pm 1}\rangle[-2].$$

Furthermore, the Frobenius action inherited by the degree $2$ part
of the pushforward is equal to $p$ times the usual Frobenius on
$\cp^\dagger\langle z^{\pm 1}\rangle$, i.e. $z\mapsto z^p$.
\end{lemma}

\begin{proof}
To check the first claim we must examine the fiberwise de Rham
complex of $\cp^\dagger\langle t,x,y\rangle$, namely the complex
$\cp^\dagger\langle t,x,y\rangle[dx,dy]$ (where $dx$ and $dy$ are
considered as odd variables) with the differential $(\partial_x+\pi
ty)\cdot dx+(\partial_y+\pi tx)\cdot dy$.  Because we care about the
$D$-module structure, keeping track of  $t$ and $\partial_t+\pi xy$
action is essential.

Beginning with $\partial_x+\pi ty$ we see that it has no kernel, and
the cokernel is isomorphic to $\frac{\cp^\dagger\langle t^{\pm
1},y^{\pm 1}\rangle}{\cp^\dagger\langle t,y\rangle}$ with
$\partial_y+\pi tx$ and $\partial_t+\pi xy$ acting via $\partial_y$
and $\partial_t$ respectively.  Considering the
$\partial_y$-equivariant short exact sequence $$0\rightarrow\cp^\dagger\langle
t,y\rangle\rightarrow\cp^\dagger\langle t^{\pm 1},y^{\pm
1}\rangle\rightarrow\frac{\cp^\dagger\langle t^{\pm 1},y^{\pm
1}\rangle}{\cp^\dagger\langle t,y\rangle}\rightarrow 0$$ it is clear that the
inclusion of the subcomplex $$\frac{\cp^\dagger\langle t^{\pm
1}\rangle}{\cp^\dagger\langle
t\rangle}1\stackrel{0}{\rightarrow}\cp^\dagger\langle t^{\pm
1}\rangle\frac{1}{y}$$ into the complex
$$\frac{\cp^\dagger\langle t^{\pm 1},y^{\pm
1}\rangle}{\cp^\dagger\langle
t,y\rangle}\stackrel{\partial_y}{\longrightarrow}\frac{\cp^\dagger\langle
t^{\pm 1},y^{\pm 1}\rangle}{\cp^\dagger\langle t,y\rangle}$$ is a
quasi-isomorphism, verifying the first claim.

An explicit map giving the isomorphism of $D$-modules between
$\cp^\dagger\langle z^{\pm 1}\rangle$ and the degree $2$ part of the
pushforward is given by
\begin{align*}
\cp^\dagger\langle z^{\pm
1}\rangle&\stackrel{\varphi}{\longrightarrow}\frac{\cp^\dagger\langle
t,x,y\rangle}{\genfrac{}{}{0pt}{}{\partial_x+\pi ty}{\partial_y+\pi tx}}dx dy\\
z^\alpha&\mapsto t^{\alpha+1}dx dy\\
z^{-i-1}&\mapsto \frac{(-\pi xy)^i}{i!}dx dy
\end{align*} where $\alpha, i\geq 0$.

All that remains is to check the $Fr$ compatibility, we do this in
two parts.  Part one:
\begin{align*}
Fr(\varphi(z^\alpha))&=Fr(t^{\alpha+1}dx dy)\\
&=\sum A_n(txy)^n t^{(\alpha+1)p}dx^p dy^p\\
&=p^2\sum A_n (txy)^{n+p-1} t^{\alpha p+1} dx dy\\
&=p^2\sum A_n (-\partial_y/\pi)^{n+p-1} y^{n+p-1}t^{\alpha p+1} dx dy\\
&=p^2\sum A_n(-1/\pi)^{n+p-1}(n+p-1)! t^{\alpha p+1}dx dy\\
&=-p\sum A_n (-1/\pi)^{n}(n+p-1)! t^{\alpha p+1}dx dy\\
&=p t^{\alpha p+1}dx dy\\
&=\varphi(pz^{\alpha p})\\
&=\varphi(pFr(z^\alpha)).
\end{align*} And part two:
\begin{align*}
Fr(\varphi(z^{-i-1}))&=Fr(\frac{(-1)^i\pi^i(xy)^i}{i!}dx dy)\\
&=\sum A_n(txy)^n \frac{(-1)^i\pi^i(xy)^{ip}}{i!}dx^p dy^p\\
&=p^2\sum A_n(txy)^n \frac{(-1)^i\pi^i(xy)^{ip+p-1}}{i!} dx dy\\
&=(p^2(-\pi)^i/i!)\sum A_n (-\partial_y/\pi)^{n} y^n (xy)^{pi+p-1} dx dy\\
&=(p^2(-\pi)^i/i!)\sum A_n(-1/\pi)^{n}\frac{(n+pi+p-1)!}{(pi+p-1)!} (xy)^{pi+p-1}dx dy\\
&=\frac{p^2(-\pi)^i}{i!(pi+p-1)!}(-1)^{i+1}p^i i! (xy)^{pi+p-1} dx dy\\
&=p (-1)^{(i+1)p-1}\frac{\pi^{(i+1)p-1}}{((i+1)p-1)!}(xy)^{(i+1)p-1}dx dy\\
&=\varphi(pz^{-(i+1)p})\\
&=\varphi(pFr(z^{-i-1})).
\end{align*}  Note that the formula (\ref{theformula}) is crucial in
the above computations.
\end{proof}

Recall that $P=P(z_i,w_j^{\pm 1})$ is a  polynomial  defining a smooth
hypersurface $X$ in $\A^n\times(\A^\times)^m$.  Denote by $e^{\pi Pt}\otimes\delta_{t=0}$
the cokernel of the map $$\cp^\dagger\langle z_i,w_j^{\pm 1},t\rangle e^{\pi
Pt}\rightarrow \cp^\dagger\langle z_i,w_j^{\pm 1},t^{\pm 1}\rangle e^{\pi Pt}.$$

\begin{lemma}\label{delta}
The $D$-module pushforward of $e^{\pi Pt}\otimes\delta_{t=0}$ onto
$\{z_i,w_j^{\pm 1}\}$ is isomorphic to $\cp^\dagger\langle z_i,w_j^{\pm 1} \rangle[-1]$.

\end{lemma}

\begin{proof}
By definition, we need to show that $\partial_t+\pi P$ acting on
$e^{\pi Pt}\otimes\delta_{t=0}$ has no kernel, and its cokernel is
isomorphic to $\cp^\dagger\langle z_i,w_j^{\pm 1} \rangle$.

The first statement follows from a direct calculation, more precisely, if
$(\partial_t+\pi P)\sum f_\alpha(z_i,w_j^{\pm 1})t^\alpha$ belongs to
$\cp^\dagger\langle z_i,w_j^{\pm 1},t\rangle$ then $\sum
f_\alpha(z_i,w_j^{\pm 1})t^\alpha\in \cp^\dagger\langle z_i,w_j^{\pm 1},t\rangle$, using
that $P$ is not a zero divisor.

For the second statement, consider the map $$\cp^\dagger\langle z_i,w_j^{\pm 1}
\rangle\rightarrow \frac{e^{\pi Pt}\otimes\delta_{t=0}}{\partial_t+\pi
P}$$ mapping $f(z_i,w_j^{\pm 1})$ to $f/t$.  One checks that it is a map of
$D$-modules and is surjective.  Since $1/t$ can be seen to be a
non-zero element of $\frac{e^{\pi Pt}\otimes\delta_{t=0}}{\partial_t+\pi
P}$, the map is not identically $0$;  then it is an isomorphism. The
conclusion follows.
\end{proof}

\begin{remark}
Our notation $e^{\pi Pt}\otimes\delta_{t=0}$ for the $D$-module above quite correctly suggests that it is a twisted (by $e^{\pi Pt}$) version of the delta functions $D$-module $\delta_{t=0}=\cp^\dagger\langle t^{\pm 1}\rangle/\cp^\dagger\langle t\rangle$.  We follow the convention of neglecting to mention the variables in which the $D$-module is just the $D$-module of functions (overconvergent or not); thus $\delta_{t=0}$ should really be equal to $\cp^\dagger\langle z_i,w_j^{\pm 1},t^{\pm 1}\rangle /\cp^\dagger\langle z_i,w_j^{\pm 1},t\rangle$.  The content of the Lemma above is that the twisting in this case can be ignored, i.e., $$e^{\pi Pt}\otimes\delta_{t=0}\cong \delta_{t=0}.$$

\end{remark}

\subsection{An ``explanation" of what is happening}\label{explanation}
We provide here an intuitive explanation for the Lemmas above and their place in the proof of the reductions from $xy+P=0$ to $P=0$.

It is well known that de Rham cohomology is closely related to
integration\footnote{See \cite{homoint} for a discussion in the physical context.}.  For example, we can integrate top forms on a smooth,
compact, oriented manifold $M$ over all of $M$ itself; this gives a
map from $\Omega^{top}M$ to $\mathbb{R}$.  By Stokes' theorem this
map factors through the de Rham cohomology $H^{top}_{dR}(M)$ since
the integrals of exact forms vanish.  If the manifold is connected,
then $H^{top}_{dR}(M)$ is one-dimensional and so the quotient map
$\Omega^{top}M\rightarrow H^{top}_{dR}(M)$ can be viewed as
integration (not normalized).

This analogy can be pursued further.  For $f\in\R[z_1,...,z_n]$
consider the $D$-module $e^f$ as above. Under suitable assumptions, it is
natural to interpret $\int_{\R^n}g(z_i)e^{f}$ as the image of
$g(z_i)dz_1...dz_n$ in $H^n_{dR}(e^f)\simeq
R[z_i]/\{\partial_{z_i}+(\partial_{z_i}f)\}$.  It is clear that in
general $H^\bullet_{dR}(e^f)$ need not be concentrated in degree
$n$, neither does it have to be one-dimensional.  Nevertheless, an
intuitive integral calculus of $D$-modules (for lack of a better
expression) can be developed that parallels closely the methods
involved in the theory of $D$-modules itself.  This calculus can be
used as an approximation and motivation to rigorous arguments which
can be extracted from it.

We remark that by interpreting integrals as appropriate de Rham
cohomologies we obtain a notion of integration in a purely algebraic
manner that transports easily to settings other than that of real or
complex numbers.  The presence of lower cohomology groups, i.e.
below the top one, in a sense represents singularities.  

\begin{definition}
For $M$ a $D$-module (or a complex of $D$-modules) on $X\times Y$ we write $$\int_X M$$ to denote the complex of $D$-modules on $Y$ obtained as the relative de Rham complex of  $M$ along the projection map $X\times Y\rightarrow Y$.  If $X=\A^n$ with coordinates $z_i$ we write $\int_{z_i} M$ instead, similarly if $X=(\A^\times)^m$ with coordinates $w_j$ we write $\int_{w^{\pm 1}_j} M$.
\end{definition}

\begin{remark}
The notation above is meant to be suggestive, certainly many things that it suggests are in fact true (see \cite{homoint}).
\end{remark}

Recall that for $P=P(z_i)$ by $\delta_{P=0}$ we mean the $D$-module on $\A^n$ given as the quotient $\mathcal{O}_{\A^n}[P^{-1}]/\mathcal{O}_{\A^n}$ where $\mathcal{O}_{\A^n}$ denotes functions in the variables $z_i$. Thus, for
example,  $\int_z \delta_{z=0}=\C[-1]$, which is what we would expect after integrating the delta function (recall that it is the \emph{top} degree that corresponds to the usual integral). Similarly, we have that $\int_z \mathcal{O}_{\A^1}=\C$; this indicates that the integral of the constant function $1$ over the line produces a singularity. In order to make the notation more suggestive let us be allowed to denote the $D$-module of functions (overconvergent or not) by the symbol $1$. Familiar formulas such
as $$\int_t e^{tP(z_i)}\delta_{t=0}=1[-1]$$ and $$\int_t e^{tP}=\delta_{P=0}[-1]$$
remain valid.

We now perform the calculation that ``demonstrates" the portion of the
Sec. \ref{general} that deals with the comparison of the
cohomology groups. Admittedly, it sheds no light on the
compatibility of the Frobenius action.  In the new language we must
compute  \begin{equation}\label{bigintegral}\int_{z_i,w^{\pm 1}_j,t,x,y}e^{tP+txy}\end{equation} with $P=P(z_1,...,z_n, w^{\pm 1}_1,...,w^{\pm 1}_m)$ and compare it to $\int_{z_i,w^{\pm 1}_j,t}e^{tP}$.

Observe that the integral in (\ref{bigintegral}) can be decomposed as $\int_{z_i,w^{\pm 1}_j,t}e^{tP}\int_{x,y}e^{txy}$.  Now the ``proof" of the Lemma \ref{exptxy} (whose relevance is now clear) proceeds as follows. Recall that $$\int_{x,y}e^{txy}=\int_{y}\delta_{ty=0}[-1].$$ It is not hard to see that $\delta_{ty=0}$ fits into an exact sequence $$0\rightarrow\delta_{t=0}\oplus\delta_{y=0}\rightarrow\delta_{ty=0}\rightarrow\delta_{t=0}\delta_{y=0}\rightarrow 0$$ so that roughly speaking $\delta_{ty=0}=\delta_{t=0}+\delta_{y=0}+\delta_{t=0}\delta_{y=0}$ and $$\int_{y}(\delta_{t=0}+\delta_{y=0}+\delta_{t=0}\delta_{y=0})=\delta_{t=0}+1[-1]+\delta_{t=0}[-1].$$ And the ``proof" is complete.

Returning to the integral (\ref{bigintegral}) we see that it splits into \begin{equation}\label{splitintegral}\int_{z_i,w^{\pm 1}_j,t}\delta_{t=0}[-1] +\int_{z_i,w^{\pm 1}_j,t}e^{tP}[-2] +\int_{z_i,w^{\pm 1}_j,t}\delta_{t=0}[-2] \end{equation}

The middle term in (\ref{splitintegral}) is exactly what we need, and it remains to be seen if the other two terms can be controlled.  We point out that $\int_{z_i,w^{\pm 1}_j,t}\delta_{t=0}=\int_{z_i,w^{\pm 1}_j}1[-1]=\int_{w^{\pm 1}_j}1[-1]=(1+1[-1])^m[-1]$.  So that for $m$ small we can state some meaningful results.

\subsection{Another look at the Key Lemma}
We observe that one can replace $xy$ by $z^2-w^2$.  Thus just as above we are
interested in understanding $\int_x e^{tx^2}$ as a complex of
$D$-modules on the line with the coordinate $t$.  In fact it is easy
to see directly that the complex is quasi-isomorphic to a single
$D$-module sitting in degree $1$, namely $\cp[x,t]/(\partial_x+2xt)$
with $(t,\partial_t)$ acting via $(t,\partial_t+x^2)$.  This
$D$-module has a $\cp$ basis consisting of $...,x^3, x^2, x, 1, t,
t^2, t^3,...$.  Analysis of the eigenspaces of $t\partial_t$ gives
the rest.  Namely we see that $$t\partial_t:x^i\mapsto
-\frac{i+1}{2} x^i$$ and $$t\partial_t:t^i\mapsto
(i-\frac{1}{2})t^i.$$  Thus as a $D$-module $$\int_x
e^{tx^2}=\left(\delta_{t=0}\oplus e^{\frac{1}{2}\log(t)}\right)[-1],$$
where the first summand is spanned by $x^i$ for $i$ odd, and the
second is spanned by the remaining basis elements.  Note that we use
the shorthand $e^{\frac{1}{2}\log(t)}$ for the $D$-module structure
on $\cp[t^{\pm 1}]$ given by $\partial_t+\frac{1}{2}/t$.

This gives another  computation of $\int_{xy}e^{txy}$ as a complex of
$D$-modules\footnote{We point out that unlike Sec. \ref{explanation} the computation here is another proof of Lemma \ref{exptxy}.}.  Namely we have the following chain of isomorphisms:
\begin{align*}
\int_{xy}e^{txy}&=\int_{zw}e^{tz^2-tw^2}\\
&=\int_w e^{-tw^2}\int_z e^{tz^2}\\
&=\int_w e^{-tw^2}\left(\delta_{t=0}\oplus
e^{\frac{1}{2}\log(t)}\right)[-1]\\
&=\delta_{t=0}[-1]\oplus e^{\frac{1}{2}\log(t)}[-1]\otimes\int_w
e^{-tw^2}\\
&=\delta_{t=0}[-1]\oplus
e^{\frac{1}{2}\log(t)}[-1]\otimes\left(\delta_{t=0}\oplus
e^{\frac{1}{2}\log(t)}\right)[-1]\\
&=\delta_{t=0}[-1]\oplus\cp[t^{\pm 1}][-2]
\end{align*} since $e^{\frac{1}{2}\log(t)}\otimes\delta_{t=0}=0$ and $e^{\frac{1}{2}\log(t)}\otimes e^{\frac{1}{2}\log(t)}=\cp[t^{\pm
1}]$ with the usual $D$-module structure.

\subsection{The case of $Y=\A^n$}
Let $P$ be a polynomial in $z_i$, i.e., an algebraic function on $\A^n$.  In this case the relationship between the hypersurfaces specified by $P+xy=0$ and $P=0$ is very simple.

\begin{prop}\label{reduction} Let $P$ be as above.
 Consider the hypersurface $\tilde{X}$ in $\A^{n+2}$
defined by $P+xy$.  Then, $$H^i(X)\cong
H^{i+2}(\tilde{X})$$ for $i\geq 1$.

Furthermore, $$p\cdot Fr_{X}=Fr_{\tilde{X}}.$$

\end{prop}
\begin{proof}
Recall that $H^\bullet(\tilde{X})[-2]\cong H^\bullet_{DW}(\tilde{X})$ and the latter is, by definition, the de Rham
cohomology of the $D$-module $e^{\pi(Pt+txy)}$.  Since $\pi Pt$ does
not depend on $x$ and $y$, the pushforward of $e^{\pi(Pt+txy)}$ onto
the remaining coordinates is isomorphic  to
$\cp^\dagger\langle z_i,t^{\pm 1}\rangle e^{\pi Pt}[-2]\oplus e^{\pi
Pt}\otimes\delta_{t=0}[-1]$ by Lemma \ref{exptxy}.\footnote{This is a $D$-module incarnation
of the constant multiple rule of basic Calculus.}

The pushforward contains as a submodule $\cp^\dagger\langle
z_i,t\rangle e^{\pi Pt}[-2]$, whose de Rham cohomology computes
$H^\bullet_{DW}(X)[-2]$.  The quotient is isomorphic to $e^{\pi
Pt}\otimes\delta_{t=0}[-2]\oplus e^{\pi Pt}\otimes\delta_{t=0}[-1]$.  Thus,
using Lemma \ref{delta},
$$H^\bullet_{DW}(X)[-2]\rightarrow
H^\bullet_{DW}(\tilde{X})\rightarrow \cp[-3]\oplus\cp[-2]$$ forms a
long exact sequence; the first statement of the Proposition then
follows immediately.

For the second statement, observe that  the
inclusion of $\cp^\dagger\langle z_i,t\rangle e^{\pi Pt}[-2]$ into
$\cp^\dagger\langle z_i,t^{\pm 1}\rangle e^{\pi Pt}[-2]$ intertwines
$pFr$ and $Fr$ by Lemma \ref{exptxy}, where the latter Frobenius descends from
$e^{\pi(tP+txy)}$.

\end{proof}

\subsection{General case}\label{general}
Consider the case of our most general $Y=\A^n\times(\A^{\times})^m$, i.e., $P=P(z_i,w_j^{\pm 1})$.  Then by applying the methods of the previous section we obtain the following description of the cohomology of $\widetilde{X}$: \begin{equation}\label{generalform}H^\bullet(\widetilde{X})\cong H^\bullet(Y)\oplus\overline{H^\bullet(X)}[-2]\end{equation} where $\overline{H^\bullet(X)}$ fits into a long exact sequence \begin{equation}\label{extension}H^\bullet(X)\rightarrow\overline{H^\bullet(X)}\rightarrow H^\bullet(Y)[1].\end{equation}  Furthermore, the map $$\alpha:H^\bullet(X)[-2]\rightarrow H^\bullet(\widetilde{X})$$ that we obtain from the above considerations is compatible with the Frobenius map in the sense that $$Fr\circ\alpha=p\cdot\alpha\circ Fr.$$

One can say more though we will not need it.  Namely, the map $H^\bullet(\widetilde{X})\rightarrow H^\bullet(Y)$ obtained from (\ref{generalform}) commutes with the Frobenius map, and the map $\beta:H^\bullet(\widetilde{X})\rightarrow H^\bullet(Y)[-1]$ obtained from (\ref{extension}) satisfies $p\cdot Fr\circ\beta=\beta\circ Fr$.

\subsection{The case of $Y=(\A^\times)^2$}
As we have seen above, in the case of $Y=\A^n$ we have a complete description of the relationship between the two hypersurfaces, whereas in the general case we can only give a partial answer.  We will now concentrate on the case of a two dimensional $Y$ and furthermore assume that it is $(\A^\times)^2$.  The answer will not be as straightforward as in the simplest case where the cohomology of $Y$ itself is so trivial that it does not contribute to the relationship.  Here we will have some contribution of the cohomology of $Y$ that we will spell out below.  We then apply our results to the analysis of a particular example that is often examined in the physics literature.

The setting we are currently working in is as follows. The Calabi-Yau $CY\subset\A^2\times(\A^{\times})^2$ is given by the equation
$xy+P(u,v)=0$ where $u$ and $v$ are coordinates on the torus
$(\A^{\times})^2$.  Let $C=\{P(u,v)=0\}$ with
$i:C\hookrightarrow(\A^{\times})^2$ denoting the inclusion.  Applying the discussion of the Sec. \ref{general} (and some explicit computations and re-scaling) we
immediately obtain the exact sequence below (it is of course part of
a larger one, but this is the useful piece):
\begin{equation}\label{sequence}H^1((\A^{\times})^2)\rightarrow^{\!\!\!\!\!\!i^*} H^1(C)
\rightarrow H^3(CY)\rightarrow^{\!\!\!\!\!\!\int}
H^2((\A^{\times})^2)\rightarrow 0\end{equation} where
$$\int:\omega=\frac{dxdudv}{xuv}\mapsto 1=\frac{dudv}{uv}$$  and the
Frobenius map $Fr$ on $H^3(CY)$ commutes with the maps $p\cdot Fr$ on the
other components of the diagram (thus for example $\int
Fr(\omega)=p^3$).

We point out that $H^1((\A^{\times})^2)$ is a two-dimensional vector space with basis $\left\{\frac{du}{u},\frac{dv}{v}\right\}$ and $H^2((\A^{\times})^2)$ is one-dimensional and spanned by $\frac{dudv}{uv}$.  The form $\omega$ above is obtained from the standard holomorphic nowhere vanishing top differential form $\frac{dx du dv}{\partial_y(xy+P(u,v))}$ by dividing it by $uv$.  It is immediately clear that $Fr(\omega)=p^3$.

If $P=P_\lambda$, i.e. it depends on a parameter, then the maps are
compatible with the Gauss-Manin connections.  Note that
$\nabla_{GM}$ on $H^i((\A^{\times})^2)$ is easy to describe\footnote{The description of the Gauss-Manin connection on the cohomology of the Calabi-Yau $CY$ can be given very explicitly if one uses the Dwork model for the cohomology.  See  the end of  Sec. \ref{gd-method}.}, namely
the generators $\frac{du}{u}$, $\frac{dv}{v}$ of
$H^1((\A^{\times})^2)$ and $\frac{dudv}{uv}$ of
$H^2((\A^{\times})^2)$ are flat.  In particular we see that
$$\int\nabla_{GM}^{i>0}\omega=\nabla_{GM}^{i>0}\int\omega=\nabla_{GM}^{i>0}1=0$$ and so $\nabla_{GM}^{i>0}\omega$
come from $H^1(C)$.

\subsubsection{Example} Let us consider a particular function $P$ on $(\A^\times)^2$, namely, \begin{equation}\label{exampleeq}P(u,v)=u+v+\lambda u^{-1}v^{-1}+1\end{equation} where $\lambda$ is a small and non-zero parameter\footnote{So that we consider $H^3(CY)$ as a vector bundle on a small punctured disk rather than a single cohomology group.}.  The hypersurface $P(u,v)=0$ is an elliptic curve
with 3 punctures aligned with the axis of $\A^2$. In the exact sequence (\ref{sequence}) above the cocycles produced by the punctures in $H^1(C)$ are annihilated 
by the image of $H^1((\A^{\times})^2)$. Thus the complete smooth elliptic curve of
genus one $\overline{C}$ that is the projective completion of $C$ makes an appearance.  More precisely, we have a refinement of the sequence (\ref{sequence}) below: $$0\rightarrow H^1(\overline{C})
\rightarrow H^3(CY)\rightarrow^{\!\!\!\!\!\!\int}
\cp\rightarrow 0.$$

If $\omega$ is a holomorphic nowhere vanishing three-form on $CY$ as above, then an explicit computation in the Dwork model for the cohomology shows that \begin{equation}\label{omegaderivative}\nabla_{\lambda\partial_\lambda}\omega\in H^{1,0}(\overline{C})\end{equation} where $H^{1}(\overline{C})=H^{1,0}(\overline{C})\oplus H^{0,1}(\overline{C})$ is the usual Hodge decomposition.

Emulating \cite{sv1} (which deals with the case of a compact three-fold) one may choose a framing of $H^1(\overline{C})$, let us call it $\{e,l\}$ satisfying some important conditions.  Namely, $$<e,l>=1,$$ where $<\cdot,\cdot>$ denotes the symplectic pairing on $H^1(\overline{C})$, $$e\in H^{1,0}(\overline{C}),$$ $$\nabla_{\lambda\partial_\lambda}e=g(\lambda)l$$ with $g$ a function of $\lambda$ and finally $$\nabla_{\lambda\partial_\lambda}l=0.$$

By the Equation (\ref{omegaderivative}) we see that $$\nabla_{\lambda\partial_\lambda}\omega=f(\lambda)e$$ where $f$ is another function of $\lambda$.  

The key difference in what remains is that we must replace the symplectic pairing on $H^3$ (used in \cite{sv1}), which is absent in a non-compact case such as ours, by the pairing between $H^3$ and $H^3_c$.  The analysis of the functions $f(\lambda)$ and $g(\lambda)$, using the Frobenius map and its compatibility with the pairing, yields certain integrality results modulo the knowledge of the behavior of the Frobenius map at the boundary point $\lambda=0$.  This extra data was obtained in \cite{vologodsky} using the theory of motives.  In \cite{shapiro} this analysis was performed explicitly for the case of the mirror quintic.

One may try to use the methods presented here to obtain integrality results also in other examples, or perhaps classes of examples.  We point out however that the problem of calculating the Frobenius map at the boundary point is a separate issue that is not at all trivial.  One can attempt to use explicit computations as was done in \cite{shapiro}, or perhaps the theory of motives can not be avoided if we ask for a general enough answer.

\section{Appendix}

\subsection{The Griffiths-Dwork method}\label{gd-method}
We explain the modifications necessary to define the
Frobenius action on the (middle dimensional) cohomology of a hypersurface inside the projective space.  We
assume that the hypersurface $V\subset\proj^n$ (for $n$ even, that is the dimension of $V$ is odd) is cut out by a homogeneous polynomial $f$ of degree
$d$, i.e., a section of the line bundle
$\mathcal{O}_{\mathbb{P}^n}(d)$ over $\mathbb{P}^n$.  This appendix is inspired by the similarly named section of \cite{coxkatz}.  \emph{Note that the Frobenius map obtained via this method is equal to $p^2$ times the usual Frobenius map.}

The idea is that one can naturally identify the middle dimensional cohomology of $V$ (for $V$ odd dimensional) with $H^n(\proj^n-V)$.  Since the space $\proj^n-V$ is affine, the latter cohomology group may be computed as the quotient of the space of top differential forms.  These have a nice homogeneous description which is readily shown to be isomorphic to an expression very similar to the one used to define the Frobenius map for a hypersurface in $\A^n$.  Some details are below.

Consider the short exact sequence of $D$-modules on $\proj^n$: $$0\rightarrow\mathcal{O}_{\proj^n}\rightarrow\mathcal{O}_{\proj^n-V}\rightarrow\delta_{V}\rightarrow 0$$ and note that the de Rham cohomology of $\delta_{V}$ computes the cohomology of $V$ up to a shift.  From a long exact sequence on cohomology (induced by the above) we obtain a map $H^n_{dR}(\mathcal{O}_{\proj^n-V})\rightarrow H^{n}_{dR}(\delta_{V})$ which in other words is $$H^n(\proj^n-V)\rightarrow H^{n-1}(V).$$  For $n$ even the above map is surjective simply because $H^{n+1}(\proj^n)=0$, however it is also true, see \cite{grif}, that it is an isomorphism.

By the above it suffices to compute  $H^n(\proj^n-V)$.  However this is not difficult  as is mentioned previously, namely it is given by the top differential
forms on $\proj^n-V$ modulo some cohomological relations.  

More generally, the de Rham complex of $\proj^n-V$ can be realized as a certain
subcomplex of $DR\,\cp[x_0,...,x_n,f^{-1}]$ (i.e., the de Rham
complex of $\A^{n+1}$ with singularities allowed along $V$).  This subcomplex can be described explicitly as
consisting of forms $\omega$ such that $L_E\omega=0$ and
$\iota_E\omega=0$ where $E$ is the Euler vector field on $\A^{n+1}$,
i.e. $E=\sum x_i\partial_i$. As usual $L$ and $\iota$ denote the Lie derivative and the contraction with a vector field respectively.  Note that since $f$ is homogeneous, so $DR\,\cp[x_0,...,x_n,f^{-1}]$ is graded and the condition that $L_E\omega=0$ is equivalent to the requirement that $\omega$ be of homogeneous degree $0$.

It is easy to see directly that the above conditions do indeed define a subcomplex, i.e., it is preserved by the de Rham differential $d$.  Furthermore, the subcomplex is clearly preserved by $\iota_{x_i\partial_j}$ and thus by  $L_{x_i\partial_j}$ since
$L_\xi=d\circ\iota_\xi+\iota_\xi\circ d$.  This means that for $\omega$ as above, the expression $L_{x_i\partial_j}\omega$ (being $0$ in the cohomology) gives a relation on the differential forms.  In fact these are the only relations that we will need.

Let us return to $H^n(\proj^n-V)$.  Observe that $\Omega=\sum(-1)^i x_i
dx_0..\widehat{dx_i}..dx_n$ gives the unique up to scalars section
of $\Omega^n_{\proj^n}\otimes\mathcal{O}(n+1)$. More precisely, the expression for $\Omega$ is an $n$-form in $DR\,\cp[x_0,...,x_n,f^{-1}]$ of homogeneous degree $n+1$ and one can check that it is annihilated by $\iota_E$.  Thus it is a section of $\Omega^n_{\proj^n}\otimes\mathcal{O}(n+1)\cong\mathcal{O}_{\proj^n}$ and so is unique. We conclude that $H^n(\proj^n-V)$ is a quotient of $$\left\{\dfrac{x^I\Omega}{f^{k+1}}: n+1+|I|=(k+1)d\right\}.$$

Let us derive the relations on the above.  Note that $L_{x_i\partial_j}\Omega=\delta_{ij}\Omega$ and we can use this to derive the equations $$\dfrac{(\partial_j x^I)\Omega}{f^k}=-k\dfrac{x^I(\partial_j f)\Omega}{f^{k+1}}.$$

We need to compare the above with the twisted algebraic de Rham complex $DR\,
\cp[x_0,...,x_n,t]e^{\pi tf}$ or rather its homogeneous degree zero $0$ part\footnote{
We note that $x_i$ have degree $1$, while $t$ has degree $-d$, thus
$\cp[x_0,..,x_n,t]_0$ is a realization of the global functions on
the total space of $\mathcal{O}(-d)$.}.  The reason for the comparison is that the latter is isomorphic to the overconvergent complex $\left(DR\,
\cp^\dagger[x_0,...,x_n,t]e^{\pi tf}\right)_0$ which carries a natural Frobenius action as described below.

It turns out that $$H^n(\proj^n-V)\cong \left(DR^{n+2}\,
\cp[x_0,...,x_n,t]e^{\pi tf}\right)_0.$$  To see this note that the latter is a quotient of $$\left\{x^I t^k dx_0..dx_n dt: n+1+|I|=(k+1)d \right\}$$ by the relations $$[(\partial_t+\pi f)\cdot x^I t^k]dx_0..dx_n dt=0$$ and $$[(\partial_j + \pi t(\partial_j f))\cdot x^I t^k]dx_0..dx_n dt=0.$$  So that $x^I t^k dx_0..dx_n dt$ corresponds up to an appropriate scalar to $\dfrac{x^I\Omega}{f^{k+1}}$. 

For
the cases we are interested in  we
have that $d=n+1$, thus $$dx dt\mapsto
\dfrac{\Omega}{f}$$ and more explicitly, the cohomology group $H^{n-1}(V)$ is identified with the vector space spanned\footnote{The meaning of spanned depends on which version of the de Rham complex we consider.  Ultimately both the algebraic and overconvergent versions give the same answer for $H^{n-1}(V)$, but the overconvergent version allows for an explicit definition of the Frobenius action.} by $x^I t^\alpha$
with $|I|=\alpha d$ modulo the relations $(\partial_i
x^I)t^\alpha=-\pi(\partial_i f)x^I t^{\alpha+1}$ and $x^I
(\partial_t t^\alpha)=-\pi f x^I t^\alpha$.

Recall that the Frobenius action on $DR\, \cp^\dagger[x_0,...,x_n,t]e^{\pi
tf}$, which preserves the subcomplex $\left(DR\, \cp^\dagger[x_0,...,x_n,t]e^{\pi
tf}\right)_0$, is given by $$\omega(x,t)\mapsto e^{\pi(t^p
f(x^p)-tf)}\omega(x^p,t^p).$$  If as in our case, $f$ depends on a
parameter $\psi$ then $\psi$ is raised to the power $p$ by the
Frobenius as well, i.e., $$\omega(\psi,x,t)\mapsto e^{\pi(t^p
f(\psi^p,x^p)-tf)}\omega(\psi^p,x^p,t^p).$$  Furthermore, if $f$ is a sum of monomials
$x^{I_j}$, we can rewrite $e^{\pi(t^p f(x^p)-tf)}$ as the product
$\Pi_j A(tx^{I_j})$ where $e^{\pi(z^p-z)}=A(z)$.

When $f$ depends on a
parameter $\psi$ we get a family of hypersurfaces $V_\psi$ whose cohomology groups possess what is known as the Gauss-Manin connection which can be described explicitly in our setting.  Namely, we have identified $H^{n-1}(V_\psi)$ with a quotient of $\left\{\psi^s x^I t^k dx_0..dx_n dt: n+1+|I|=(k+1)d \right\}$ and the connection is given by $$\nabla_{\partial_\psi}=\partial_\psi+\pi t(\partial_\psi f).$$

One can show that the Frobenius map is compatible with the Gauss-Manin connection in the sense that $$\nabla_{\psi\partial_\psi}\circ Fr=p Fr\circ \nabla_{\psi\partial_\psi}.$$

\subsection{A sketch of the algebraic Dwork-de Rham comparison}
We use the language of $D$-modules and their de Rham cohomology to relate the algebraic (i.e., omitting the analytic aspects of overconvergence) Dwork cohomology of a hypersurface to its de Rham cohomology. We are working up to a shift of the cohomological degree.  The discussion presented here is similar to \cite{ddr}.

Let $f:Y\rightarrow \A^1$ and consider the maps $$\A^1\times X\stackrel{Id\times
f}{\rightarrow} \A^1\times\A^1\stackrel{p_1}{\rightarrow} \A^1$$
$$(\tau,x)\mapsto(\tau,f(x))\mapsto \tau$$  and let $e^{\tau f(y)}$
be a $D$-module on $\A^1\times Y$ as before.  Denote by $p_{\A^1
}$ the composition of the two maps above, then
$p_{\A^1
*}e^{\tau f(y)}\simeq p_{1
*}\circ(Id\times f)_* e^{\tau f(y)}\cong p_{1
*}\circ(Id\times f)_*\circ(Id\times f)^* e^{\tau t}\cong p_{1 *}(e^{\tau t}\otimes(Id\times f)_*
\mathcal{O}_{\A^1\times Y})\cong p_{1 *}(e^{\tau t}\otimes
\mathcal{O}_{\A^1}\boxtimes f_*(\mathcal{O}_Y))\cong
\widehat{f_*(\mathcal{O}_Y)}$ where $\widehat{f_*(\mathcal{O}_Y)}$
denotes the Fourier transform of $f_*(\mathcal{O}_Y)$.

The above can be used to extract another proof of the fact that the
total de Rham cohomology of the $D$-module $e^{\tau f}$ computes (up to
shift) the cohomology of the $0$-set of $f$ (when $f$ is smooth,
otherwise it is a better version of the cohomology of $f=0$).
Namely, the de Rham cohomology of $e^{\tau f}$ is isomorphic to that
of $\widehat{f_*(\mathcal{O}_Y)}$, and the cohomology of
$\widehat{f_*(\mathcal{O}_Y)}$ is equal to $i^!_0 f_*(\mathcal{O}_X)$ (by definition of the de Rham cohomology and the Fourier transform) where $i_0$ comes from the diagram below.  $$\xymatrix{f^{-1}(0)\ar[d]^\pi\ar[r]^i & Y\ar[d]^f\\
0\ar[r]^{i_0}& \A^1}$$  By base change we can replace $i^!_0 f_*(\mathcal{O}_X)$ by $\pi_* i^!\mathcal{O}_Y$ which is (again up to shift) the cohomology of the zero set of $f$.

\medskip
\noindent{\bf Acknowledgments.} We are indebted to A. Schwarz for posing the initial question as well as much appreciated discussions, comments and advice.  Thanks are also due to E. Macri for useful remarks.  We appreciate the hospitality of IH\'{E}S, MPIM Bonn and
University of Waterloo where parts of this paper were written.

\medskip
\noindent Institut des Hautes \'{E}tudes Scientifiques, Bures-sur-Yvette, France
\newline \emph{E-mail address}:
\textbf{shapiro@ihes.fr}

\end{document}